\documentclass{amsart}
\numberwithin{equation}{section}

\usepackage{bbm}
\usepackage{comment}
\usepackage{verbatim}
\usepackage{mathrsfs}
\usepackage{latexsym}
\usepackage{epsfig}
\usepackage{amsmath}
\usepackage[usenames,dvipsnames]{xcolor}
\usepackage{bbm}
\usepackage{graphics}
\usepackage{amssymb}
\usepackage{amsthm}
\usepackage[alphabetic]{amsrefs}
\usepackage{amsopn}
\usepackage{amscd}
\usepackage[all,knot]{xy}
\xyoption{all}
\usepackage{rotating}
\usepackage{tikz}
\usetikzlibrary{calc}
\usepgflibrary{shapes.geometric}
\usepgflibrary{shapes.misc}
\usetikzlibrary{positioning}
\usetikzlibrary{decorations}
\usetikzlibrary{decorations.pathreplacing}
\usepackage{hyperref}

\theoremstyle{plain}
\newtheorem{thm}{Theorem}[subsection]
\newtheorem{lem}[thm]{Lemma}
\newtheorem{prop}[thm]{Proposition}
\newtheorem{cor}[thm]{Corollary}

\newtheorem*{thm*}{Theorem}
\newtheorem*{lem*}{Lemma}
\newtheorem*{prop*}{Proposition}
\newtheorem*{cor*}{Corollary}

\theoremstyle{definition}
\newtheorem{defn}[thm]{Definition}
\newtheorem*{defn*}{Definition}

\newtheorem{ex}[thm]{Example}
{}
\newtheorem{rem}[thm]{Remark}
\newtheorem*{rem*}{Remark}

{}
\newtheorem{convention}[thm]{Convention}{}
{}
\newtheorem*{ack}{Acknowledgements}{}

\theoremstyle{remark}
{}
{}
{}

\def\to{\longrightarrow}

\def\NN{\mathbb{N}}

\def\ZZ{\mathbb{Z}}

\def\AA{\mathbb{A}}
\def\PP{\mathbb{P}}
\def\FF{\mathbb{F}}
\def\str{\mathcal{O}}

\def\a{\alpha}

\def\sfD{\mathsf{D}}

\def\sfT{\mathsf{T}}

\def\sfC{\mathsf{C}}

\def\perf{\mathrm{perf}}
\def\bdd{\mathrm{b}}
\def\sg{\mathrm{sg}}

\def\dgcatk{\mathsf{DGcat}_k}

\DeclareMathOperator{\End}{End}
\DeclareMathOperator{\Spec}{Spec}

\DeclareMathOperator{\coker}{coker}

\DeclareMathOperator{\id}{id}

\DeclareMathOperator{\Hom}{Hom}
\DeclareMathOperator{\RHom}{\mathbf{R}Hom}
\DeclareMathOperator{\RuHom}{\mathbf{R}\underline{Hom}}

\DeclareMathOperator{\Ext}{Ext}

\DeclareMathOperator{\cone}{cone}

\DeclareMathOperator{\colim}{colim}

\DeclareMathOperator{\modu}{\mathsf{mod}}

\DeclareMathOperator{\smodu}{\underline{\mathsf{mod}}}

\DeclareMathOperator{\gr}{\mathsf{gr}}

\DeclareMathOperator{\qgr}{\mathsf{qgr}}

\DeclareMathOperator{\sgr}{\underline{\mathsf{gr}}}

\DeclareMathOperator{\coh}{coh}

\DeclareMathOperator{\thick}{thick}

\DeclareMathOperator{\Perf}{\mathsf{Perf}}

\DeclareMathOperator{\tors}{\mathsf{tors}}
\DeclareMathOperator{\rad}{rad}

\def\a1{\mathbb{E}}

\title{Homotopy invariants of singularity categories}
\author{Sira Gratz}
\address{
Sira Gratz, School of Mathematics and Statistics,
University of Glasgow,
University Place,
Glasgow G12 8SQ
}
\email{Sira.Gratz@glasgow.ac.uk}
\author{Greg Stevenson}
\address{Greg Stevenson, School of Mathematics and Statistics,
University of Glasgow,
University Place,
Glasgow G12 8SQ
}
\email{gregory.stevenson@glasgow.ac.uk}
\urladdr{http://www.maths.gla.ac.uk/~gstevenson/}


\begin{document}

\begin{abstract}
\noindent We present a method for computing $\mathbb{A}^1$-homotopy invariants of singularity categories of rings admitting suitable gradings. Using this we describe any such invariant, e.g.\ homotopy K-theory, for the stable categories of self-injective algebras admitting a connected grading. A remark is also made concerning the vanishing of all such invariants for cluster categories of type $A_{2n}$ quivers.
\end{abstract}

\maketitle


\setcounter{tocdepth}{1}
\tableofcontents


\section{Introduction}

Gradings often make life significantly easier. For instance, if $\Lambda$ is an exterior algebra on $n+1$ generators over a field $k$ then, although its stable category $\smodu \Lambda$ cannot even have a non-trivial t-structure, if we standard grade $\Lambda$ then the graded stable category $\sgr \Lambda$ has a full strong exceptional collection. This makes computing $\AA^1$-homotopy invariants of $\sgr \Lambda$ very easy\textemdash{}they are just given by $n+1$ copies of the corresponding invariant evaluated at the base field. However, one doesn't always want to work with graded modules. In this article we exploit work of Tabuada \cite{TabuadaA1} and Keller, Murfet, and Van den Bergh \cite{KMV} to describe the invariants of $\smodu \Lambda$ in terms of those of $\sgr \Lambda$, the graded stable category.

After covering the required preliminaries on orbit categories, $\AA^1$-homotopy invariants, and singularity categories in Section~\ref{sec:prelims} we use Tabuada's work on $\AA^1$-homotopy invariants of orbit categories to present, in Theorem~\ref{thm_main1}, a cofibre sequence relating invariants of graded and ungraded singularity categories. We then specialise to finite dimensional algebras and exploit the very strong results on existence of tilting objects for graded singularity categories to perform concrete computations. In particular, we show in Theorem~\ref{thm:Frobenius} that if $\Lambda$ is a finite dimensional self-injective $k$-algebra admitting a connected grading then, for any $\AA^1$-homotopy invariant $\a1$, we have
\begin{displaymath}
\a1(\smodu \Lambda) \cong \cone (\a1(k) \xymatrix{\ar[rr]^-{\boldsymbol{\cdot}\dim \Lambda}&&}\a1(k)).
\end{displaymath}
This generalizes the computation of $K_0(\smodu \Lambda)$ for such algebras by Tachikawa and Wakamatsu \cite{TW91}.

In the final section we discuss a special case of a result of Tabuada concerning $\AA^1$-homotopy invariants of cluster categories. In \cite{TabuadaA1}*{Corollary~2.11} a presentation for the $\AA^1$-homotopy invariants of cluster categories of finite acyclic quivers is given in terms of a cofibre sequence. Using this we point out that for the Dynkin quivers $A_{2n}$ this actually implies that all $\AA^1$-homotopy invariants of the corresponding cluster category vanish.

\begin{ack}
We are grateful to Sebastian Klein for inspiring conversations, originating from discussions on tt-Chow groups, which led us to the considerations which were the genesis of this work. We are also grateful to the anonymous referee for several helpful comments which improved the exposition. The second author thanks Lance Gurney and Shane Kelly for precious discussions.
\end{ack}



\section{Preliminaries}\label{sec:prelims}

Our main result arises from putting together several observations made by others. In this section we recall some salient details regarding the ingredients we need. This also serves to fix ideas and notation for the rest of the article.

Throughout we will work over a fixed base field, which we will denote by $k$, and by DG-category we always mean DG-category over $k$. Things could, as usual, be extended to more general base rings but we remain in the simplest case for the sake of avoiding technicalities in the exposition.


\subsection{Homotopy invariants of orbit categories}\label{sec:hoinv}

We begin with a brief review of orbit categories. For further details the reader can consult \cite{KellerOrbit}.

Let $\sfC$ be a DG-category and suppose we are given a DG-functor $F\colon C\to C$ such that $F$ is a quasi-equivalence, and hence $H^0(F)$ is an equivalence of categories. The DG-orbit category of $\sfC$ with respect to $F$, denoted by $\sfC/F$, is the DG-category whose objects are the same of those of $\sfC$ and whose morphism complexes are defined by
\begin{displaymath}
\sfC/F(c,c') = \colim_{i\in \NN}\limits (\bigoplus_{j\in \NN} \sfC(F^jc, F^i c'))
\end{displaymath}
where the transition maps are the obvious ones, namely
\begin{displaymath}
\xymatrix{
\bigoplus\limits_{j\in \NN}\sfC(F^jc, c') \ar[r]^-{\oplus F} & \bigoplus\limits_{j\in \NN}\sfC(F^jc, F^1c') \ar[r]^-{\oplus F} & \bigoplus\limits_{j\in \NN}\sfC(F^jc, F^2c') \ar[r] & \cdots
}
\end{displaymath}
The DG-category structure is uniquely induced from that of $\sfC$ using functoriality of $F$, compatibility of tensor products with colimits, and the universal property of colimits.

One can check that, upon taking the homotopy category, this gives the more familiar formula
\begin{displaymath}
H^0(\sfC/F)(c,c') = \bigoplus_{i\in \ZZ}H^0(\sfC)(c, H^0(F)^i c')
\end{displaymath}
and has the effect of making $H^0(F)$ isomorphic to the identity functor. Put a little more carefully, there is a canonical functor $\pi\colon \sfC \to \sfC/F$ together with a natural transformation $\pi \to \pi F$ which becomes an isomorphism after taking homotopy. The DG-orbit category equipped with the canonical projection functor and natural transformation as above is initial in the appropriate sense with respect to triples of such data.

\begin{rem}
Technically speaking one is stabilizing (at the DG level) and taking orbits (at the triangulated level) with respect to the action of the additive monoid $\NN$ and its group completion $\ZZ$, generated by the action of the chosen functor, and so should indicate this somehow in the notation. But, since we shall only ever work with a single functor we omit such decorations from the notation.
\end{rem}

\begin{rem}
The formula for the morphism complexes in the DG-orbit category simplifies if $F$ is an honest equivalence of DG-categories: in this case
\begin{displaymath}
\sfC/F(c,c') \cong \bigoplus_{i\in \ZZ} \sfC(c, F^ic')
\end{displaymath}
(as is always the case after taking homology).
\end{rem}

We note that even if $\sfC$ is pretriangulated, so that $H^0(\sfC)$ is triangulated, and $F$ is an honest DG-equivalence, it may no longer be the case that $\sfC/F$ is pretriangulated. Of course, one of the draws of the DG-setting is that we can just take $\Perf(\sfC/F)$ to remedy this situation, where $\Perf(\sfC/F)$ denotes the DG-category of perfect DG-modules over $\sfC/F$. This being said, it is natural at this juncture to lay our cards on the table concerning the standing hypotheses we will make about existence of cones and idempotent completeness.

\begin{convention}\label{conv:horror}
Unless explicitly mentioned otherwise we assume our DG-categories are pretriangulated with idempotent complete homotopy categories. In particular, despite the generality in which we have defined things above, for us $\sfC$ will always be quasi-equivalent to $\Perf(\sfC)$ and by $\sfC/F$ we will really mean $\Perf(\sfC/F)$. This is, without doubt, an abuse. But in our examples all Verdier quotients will be idempotent complete and our main focus is invariants which invert derived Morita equivalences and so it is a harmless abuse.
\end{convention}

On the occasions when we need to explicitly discuss idempotent completion we will use $\natural$ to denote it.

\smallbreak
\begin{center}
*\ *\ *
\end{center}

We now give a quick review of $\AA^1$-homotopy invariants. We let $\dgcatk$ denote the category of (essentially) small DG-categories over $k$, i.e.\ this is the category with objects the small DG-categories and morphisms given by isomorphism classes of DG-functors. In addition we fix some triangulated category $\sfT$.

We recall that a localization sequence of DG-categories is, essentially, the inclusion of a thick subcategory followed by the corresponding Verdier quotient (the catch being this is only up to Morita equivalence). Some further details and equivalent formulations can be found in \cite{KellerDG}*{Theorem~4.11}.

\begin{defn}
A functor $\a1\colon \dgcatk \to \sfT$ is an $\AA^1$\emph{-homotopy invariant} if:
\begin{itemize}
\item[(1)] $\a1$ sends derived Morita equivalences to isomorphisms, in particular for any DG-category $\sfC$ the canonical inclusion $\sfC\to \Perf(\sfC)$ is sent to an isomorphism by $\a1$;
\item[(2)] $\a1$ sends localization sequences of DG-categories to triangles;
\item[(3)] $\a1$ inverts the canonical inclusion
\begin{displaymath}
\sfC \to \sfC[t] = \sfC\otimes_k k[t]
\end{displaymath}
for every DG-category $\sfC$, where $k[t]$ is concentrated in degree $0$.
\end{itemize}
\end{defn}

Important examples are given by variants of $K$-theory, for instance Weibel's homotopy $K$-theory and the topological $K$-theory of DG-categories (see \cite{AHK-theory} for the latter), and by periodic cyclic homology; see \cite{TabuadaFT} for further details.

\begin{rem}
Let us stress that, in general, $K$-theory satisfies only condition (1). By forcing it to satisfy condition (2) one gets non-connective $K$-theory, and then forcing that to satisfy (3) one gets homotopy $K$-theory, which we denote by $KH$.  Throughout, we will work generally work with homotopy $K$-theory, as we will rely on work of Tabuada which in turn relies on $\AA^1$-invariance. Homotopy $K$-theory is, in an appropriate sense, universal and so we can deduce much about general $\AA^1$-invariants by understanding $KH$ and in particular $KH_0$. The problem is that $KH_0$ is more involved to compute than $K_0$. As a result our general strategy will be to make arguments allowing us to reduce to computing honest Grothendieck groups.
\end{rem}

\begin{rem}\label{rem:SOD}
One consequence of the definition is that an $\AA^1$-homotopy invariant sends semi-orthogonal decompositions to direct sum decompositions. This is already true for \emph{additive invariants}, i.e.\ functors satisfying (1), and follows from functoriality: any adjoint to a fully faithful inclusion is sent to a retract by such a functor.

It follows that if $H^0\Perf(\sfC)$ has a full exceptional collection $(E_1,\ldots, E_n)$ then
\begin{displaymath}
\a1(\sfC) \cong \a1(k)^{\oplus n}
\end{displaymath}
in the target category $\sfT$ (throughout we assume that for a collection to be exceptional each $\thick(E_i)$ is admissible).
\end{rem}

The main fact which we will need concerning $\AA^1$-homotopy invariants is that they are compatible with taking orbits.

\begin{thm}[\cite{TabuadaA1}*{Theorem~1.5}]\label{thm:tabuada}
Let $\sfC$ be a DG-category and $F\colon \sfC\to \sfC$ a quasi-equivalence. Then for any $\AA^1$-homotopy invariant $\a1\colon \dgcatk \to \sfT$ there is a distinguished triangle
\begin{displaymath}
\xymatrix{
\a1(\sfC) \ar[rr]^-{\a1(F) - \id} && \a1(\sfC) \ar[r]^-{\a1(\pi)} & \a1(\sfC/F) \ar[r] & \Sigma \a1(\sfC)
}
\end{displaymath}
where $\pi\colon \sfC\to \sfC/F$ is the canonical DG-functor.
\end{thm}

Using this theorem one can reduce computations of $\AA^1$-invariants $\a1(\sfC/F)$ to understanding the action of $\a1(F)$ on $\a1(\sfC)$. As we shall see this is often easier than trying to directly compute $\a1(\sfC/F)$.

\subsection{Graded and ungraded modules}\label{sec_graded}

In this section we recall some details on singularity categories and give a sketch of a result we will use which is due to Keller, Murfet, and Van den Bergh.

Throughout, as above, we fix a base field $k$. Let $A$ be a finitely generated noetherian graded $k$-algebra. Recall that $A$ is said to be \emph{connected} if $A$ is non-negatively graded and $A_0=k$.

We can associate with $A$ the category of finitely generated graded $A$-modules $\gr A$ and then go on to form the bounded derived category of finitely generated graded $A$-modules $\sfD^\bdd(\gr A)$ and the full subcategory of perfect complexes $\sfD^\perf(\gr A)$ within. Our convention throughout is to work with right modules.

The graded singularity category of $A$ is the quotient
\begin{displaymath}
\sfD_\sg(\gr A) = \sfD^\bdd(\gr A)/ \sfD^\perf(\gr A).
\end{displaymath}
Each of these categories comes equipped with a grading shift autoequivalence $(1)$ which is defined on graded modules by reindexing
\begin{displaymath}
M(i)_j = M_{i+j}.
\end{displaymath}

Similarly we can work with ungraded $A$-modules and define $\sfD_\sg(\modu A)$. We say that an $A$-module $M$ is \emph{gradable} if there is a graded $A$-module whose underlying module is $M$.

All of the triangulated categories mentioned above are algebraic and thus have DG-enhancements; these are not necessarily unique in the case of the singularity category, and so our convention is to work with the canonical one i.e.\ the one induced by viewing the singularity category as a localization of the bounded derived category, see \cites{Drinfeld,KellerQ}. So we have access to the definitions and tools mentioned in Section~\ref{sec:hoinv}. We will be slightly colloquial in our approach, and often use the homotopy incoherent language of triangulated categories. However, for $\AA^1$-homotopy invariants and orbit categories to make sense we require homotopy coherence and so the reader should have this in the back of her mind.

The result we wish to recall (in a slightly extended form) compares the two categories $\sfD_\sg(\gr A)$ and $\sfD_\sg(\modu A)$. There is an obvious exact comparison functor, given by forgetting the grading, 
\begin{displaymath}
F\colon \sfD_\sg(\gr A) \to \sfD_\sg(\modu A)
\end{displaymath}
which `factors' via a functor
\begin{displaymath}
\widetilde{F}\colon \sfD_\sg(\gr A) / (1) \to \sfD_\sg(\modu A)^\natural
\end{displaymath}
by the universal property of the orbit category. The reason for the scare quotes and the $\natural$ is that, in the process of forming the orbit category, we idempotent complete it and so we had better idempotent complete the target of the comparison functor (cf.\ Convention~\ref{conv:horror} and note that, in keeping with it, we drop the $\natural$ from now on). As an aside we note that in many cases, for instance if $A$ is complete, then the singularity category is already idempotent complete.

This comparison functor $\widetilde{F}$ is always an embedding.

\begin{lem}[\cite{KMV}*{Lemma~A.7}]
The functor $\widetilde{F}$ is fully faithful.
\end{lem}

Since $\widetilde{F}$ is exact (and we're conflating $\sfD_\sg(\gr A) / (1)$ with its pretriangulated hull) the next lemma is an immediate consequence.

\begin{lem}\label{lem_orbitequiv}
If a classical generating set for $\sfD_\sg(\modu A)$ is gradable then 
\begin{displaymath}
\widetilde{F}\colon \sfD_\sg(\gr A) / (1) \to \sfD_\sg(\modu A)
\end{displaymath}
is an equivalence.
\end{lem}
\begin{proof}
Since $\widetilde{F}$ is fully faithful it embeds $\sfD_\sg(\gr A)/(1)$ as a thick subcategory. If a classical generating set for $\sfD_\sg(\modu A)$ is gradable then the image of $\widetilde{F}$ contains said generating set and so, since the image of $\widetilde{F}$ is thick, it must be all of $\sfD_\sg(\modu A)$.
\end{proof}

If $A$ is connected graded, so in particular graded local, then the trivial module $k = A/A_{\geq 1}$ is always gradable and we obtain the following observation of Keller, Murfet, and Van den Bergh.

\begin{prop}[\cite{KMV}*{Proposition~A.8}]
If $A$ is a finitely generated connected commutative graded $k$-algebra such that the augmentation ideal $A_{\geq 1}$ defines an isolated singularity in $\Spec A$ then $\widetilde{F}$ is an equivalence.
\end{prop}

There are also other situations in which the lemma applies. The following elementary lemma covers some further cases of interest, for instance it applies to finite dimensional algebras with respect to a grading making the (ungraded) Jacobson radical homogeneous. This can be viewed as the obvious noncommutative generalization of the proposition in (geometric) dimension zero. 

\begin{lem}
If $A$ is finite dimensional and the simple modules are gradable then $\widetilde{F}$ is an equivalence.
\end{lem}
\begin{proof}
Since every object of $\modu A$ has a finite composition series with semisimple subquotients the simples generate $\modu A$ under finite direct sums and extensions. It follows that the simples form a classical generating set for $\sfD^\bdd(\modu A)$. As the singularity category is a quotient of the bounded derived category their images in $\sfD_\sg(\modu A)$ are thus also a classical generating set, which lifts along $F$ by hypothesis.
\end{proof}

\subsection{Koszul duality}

Let us now recall a small piece of the theory of Koszul duality which will be used in one of our applications. Fix a field $k$, as above, and let $\Lambda$ be a left and right coherent connected graded $k$-algebra. 

\begin{defn}
The graded algebra $\Lambda$ is \emph{Koszul} if the minimal graded free resolution of the trivial module $k$ is linear. Put explicitly the requirement is that if
\begin{displaymath}
\xymatrix{
\cdots \ar[r] & F_i \ar[r] & F_{i-1}\ar[r] & \cdots \ar[r] & F_0
}
\end{displaymath}
is the minimal graded free resolution then $F_i$ is generated in degree $i$.

The \emph{Koszul dual} of $\Lambda$ is
\begin{displaymath}
\Lambda^! = \Ext^*_\Lambda(k,k)
\end{displaymath}
where the Ext-algebra is computed sans grading and $\Lambda^!$ is graded using cohomological degree.
\end{defn}

\begin{rem}\label{rem:Koszul1}
This definition can be extended beyond the connected case, cf.\ Remark~\ref{rem:Koszul2}.
\end{rem}

The facts we will need concerning Koszul duality are summarized in the following theorem; these facts are all standard, and we do not attempt to give exhaustive references. At this level of generality one can consult \cite{VillaSaorin}*{4.2 and 4.3} for further details. Really all that is needed for the statement we give is that $\Ext^*_\Lambda(k,k)$ is concentrated on the diagonal with respect to the bigrading by cohomological and internal degrees, see for example \cite{PPQA}*{Chapter~2.1}.

\begin{thm}\label{thm_KD}
Suppose $\Lambda$ is a Koszul algebra. Then $\Lambda^!$ is also Koszul and $\Lambda^{!!}\cong \Lambda$. If in addition $\Lambda$ is finite dimensional then $\Lambda^!$ has finite global dimension. Moreover, in this case the full subcategory
\begin{displaymath}
\sfT = \{\Sigma^{i}k(-i) \; \vert \; i\in \ZZ\}
\end{displaymath}
of $\sfD^\mathrm{b}(\gr \Lambda)$ is tilting. It induces an equivalence of triangulated categories
\begin{displaymath}
\phi = \RHom(\sfT,-) \colon \sfD^\mathrm{b}(\gr \Lambda) \to \sfD^\mathrm{perf}(\gr \Lambda^!)
\end{displaymath}
such that
\begin{itemize}
\item $\phi\circ (1) \cong \Sigma(-1) \circ \phi$;
\item $\phi$ sends perfect complexes to complexes with torsion cohomology.
\end{itemize}
In particular, $\phi$ restricts to an equivalence
\begin{displaymath}
\sfD^\mathrm{perf}(\gr \Lambda) \to \sfD^\mathrm{perf}_\mathrm{tors}(\gr \Lambda^!),
\end{displaymath}
and so induces an equivalence
\begin{displaymath}
\sfD_\mathrm{sg}(\gr \Lambda) \to \sfD^\mathrm{b}(\qgr \Lambda^!) = \sfD^\mathrm{perf}(\gr \Lambda^!) / \sfD^\mathrm{perf}_\mathrm{tors}(\gr \Lambda^!).
\end{displaymath}
\end{thm}

\begin{rem}
In our applications $\Lambda^!$ will be coherent and so $\gr \Lambda^!$ is an abelian category and one can identify $\sfD^\mathrm{perf}(\gr \Lambda^!)$ with $\sfD^\mathrm{b}(\gr \Lambda^!)$. Moreover, $\sfD^\mathrm{b}(\qgr \Lambda^!)$ is then the bounded derived category of $\qgr \Lambda^! = \gr \Lambda^!/\tors \Lambda^!$, where $\tors\Lambda^!$ is the full subcategory of finitely presented torsion modules.
\end{rem}

An example of particular relevance is when $\Lambda$ is $\bigwedge(k(-1)^{n+1})$, an exterior algebra on $n+1$ degree $1$ generators. This algebra is certainly finite dimensional and is also Koszul, with Koszul dual $\mathrm{S}(k(-1)^{n+1})$ the symmetric algebra on $n+1$ degree $1$ generators. In this situation the theorem gives the classical BGG correspondence \cite{BGG}
\begin{displaymath}
\sgr \Lambda \cong \sfD^\mathrm{b}(\coh \mathbb{P}^{n}_k)
\end{displaymath}
sending the autoequivalence $(1)$ on the left to the autoequivalence $\Sigma \str(-1)\otimes \text{-}$ on the right.

\begin{rem}
The equivalence $\phi$ of the theorem can be, by construction, lifted to a quasi-equivalence of DG-categories. Thus the induced equivalences, on the perfect complexes and singularity categories, inherit compatible DG-categorical lifts. 
\end{rem}

\begin{rem}
If $\Lambda$ is a finite dimensional graded Frobenius algebra then $\sgr \Lambda$ and $\smodu \Lambda$ are endowed with canonical DG-enhancements by virtue of being stable categories of Frobenius categories. Indeed, they can be viewed as homotopy categories of acylic complexes of projective-injective objects. These enhancements are quasi-equivalent to those induced via localization from the pertinent bounded derived category. We are not aware if this is recorded in the literature explicitly, but it can be handily deduced from the results of \cite{KrauseStable}.
\end{rem}



\section{The main results}
We are now in a position to indicate how one can compute $\mathbb{A}^1$-homotopy invariants of singularity categories in the presence of a favourable grading.

\subsection{A general statement}

We first give the obvious statement one gets from the given ingredients.

\begin{thm}\label{thm_main1}
Let $A$ be a finitely generated noetherian graded $k$-algebra such that there is a classical generating set of gradable modules for $\sfD_\sg(\modu A)$. Then for any $\mathbb{A}^1$-homotopy invariant $\a1$ there is an isomorphism
\begin{displaymath}
\a1(\sfD_\sg(\modu A)) \cong \a1(\sfD_\sg(\gr A)/(1)),
\end{displaymath}
where $(1)$ denotes the grading shift autoequivalence on $\sfD_\sg(\gr A)$, which induces a triangle
\begin{displaymath}
\xymatrix{
\a1(\sfD_\sg(\gr A)) \ar[rr]^-{\a1(1) - \id} && \a1(\sfD_\sg(\gr A)) \ar[r] & \a1(\sfD_\sg(\modu A)) \ar[r] & \Sigma \a1(\sfD_\sg(\gr A)).
}
\end{displaymath}
\end{thm}
\begin{proof}
By the hypotheses on $A$ we see from Lemma~\ref{lem_orbitequiv} that there is an equivalence
\begin{displaymath}
\sfD_\sg(\gr A)/(1) \cong \sfD_\sg(\modu A)
\end{displaymath}
(recall that we are identifying the orbit category with its pretriangulated hull, see Convention~\ref{conv:horror}). The first statement of the theorem is immediate from this. The existence of the claimed cofibre sequence then follows from Tabuada's result Theorem~\ref{thm:tabuada}.
\end{proof}

We now specialize to finite dimensional Koszul algebras. This case is particularly nice as Koszul duality allows one to rephrase the computation in terms of orbit categories arising from noncommutative projective algebraic geometry. 


\begin{cor}\label{cor:Koszul}
Let $\Lambda$ be a finite dimensional $k$-algebra equipped with a Koszul grading and denote by $\Lambda^!$ its Koszul dual. Then for any $\mathbb{A}^1$-homotopy invariant $\a1$ there is an isomorphism
\begin{displaymath}
\a1(\sfD_\sg(\gr \Lambda)) \cong \a1(\sfD^\mathrm{perf}(\qgr \Lambda^!))
\end{displaymath}
which induces a triangle
\begin{displaymath}
\xymatrix{
\a1(\sfD^\mathrm{perf}(\qgr \Lambda^!)) \ar[rrr]^-{\a1(\Sigma\str(-1)\otimes \text{-}) - \id} &&& \a1(\sfD^\mathrm{perf}(\qgr \Lambda^!)) \ar[r] & \a1(\sfD_\sg(\modu \Lambda)) \ar[r] &. 
}
\end{displaymath}
\end{cor}
\begin{proof}
Since $\Lambda$ admits a connected grading it is local as a plain algebra. The unique simple $k$ is certainly gradable and so Theorem~\ref{thm_main1} applies to give us a cofibre sequence
\begin{displaymath}
\xymatrix{
\a1(\sfD_\sg(\gr \Lambda)) \ar[rr]^-{\a1(1) - \id} && \a1(\sfD_\sg(\gr \Lambda)) \ar[r] & \a1(\sfD_\sg(\modu \Lambda)) \ar[r] & \Sigma \a1(\sfD_\sg(\gr \Lambda)).
}
\end{displaymath}
By Theorem~\ref{thm_KD} there is an equivalence $\sfD_\sg(\gr \Lambda) \cong \sfD^\mathrm{perf}(\qgr \Lambda^!)$ which identifies $(1)$ on the former category with $\Sigma\str(-1)\otimes \text{-}$ on the latter. Rewriting the cofibre sequence above using these identifications gives the cofibre sequence claimed in the statement.
\end{proof}

\begin{rem}\label{rem:Koszul2}
One can consider a more general notion of Koszul where instead of working over $k$ with connected algebras one works with algebras augmented over more general semisimple bases. There is, of course, an analogue in this generality and all the proofs go through unchanged. 
\end{rem}

\subsection{Gorenstein algebras}

In order to use these results to effectively compute invariants one needs a handle on the graded singularity category. Fortunately, introducing a grading has the effect of adding more simples and splitting up the Exts. As a result, there are frequently semi-orthogonal decompositions at the graded level which one can exploit to perform computations, cf.\ Remark~\ref{rem:SOD}. 

The situation is particularly good for certain finite dimensional algebras.

\begin{defn}
Following \cite{BurkeStevenson}, we say a graded $k$-algebra $\Lambda$ is \emph{Artin-Schelter Gorenstein} if $\Lambda$ has finite injective dimension as both a left and a right $\Lambda$-module and
\begin{displaymath}
\RuHom(\Lambda_0,\Lambda) \cong \Sigma^d\Lambda_0(a)
\end{displaymath}
for some integers $d$ and $a$, where $\RuHom$ is the right derived functor of the graded hom-functor. We call the $a$ appearing above the \emph{Gorenstein parameter} of $\Lambda$.
\end{defn}

\begin{rem}
If $\Lambda$ is a finite dimensional graded $k$-algebra with $\Lambda_0 = \Lambda/\rad(\Lambda)$ then $\Lambda$ being Artin-Schelter Gorenstein implies that $\Lambda$ is self-injective. In particular, there is a canonical equivalence $\sfD_\sg(\gr \Lambda) \cong \sgr \Lambda$.

One can sometimes get away with asking less of $\Lambda$; by \cite{Yamaura} the stable category of graded modules over any non-negatively graded self-injective algebra $\Lambda$ with $\Lambda_0$ of finite global dimension has a tilting object. This can be used to run a similar argument, at least in some cases, to the one given below. However, we do not treat this case explicitly.
\end{rem}

\begin{cor}\label{cor:Gor}
Let $\Lambda$ be a finite dimensional basic Artin-Schelter Gorenstein $k$-algebra with $\Lambda_0 = \Lambda/\rad(\Lambda)$, where $\rad(\Lambda)$ denotes the ungraded radical, and with all simples $1$-dimensional. Then for any $\mathbb{A}^1$-homotopy invariant $\a1$ there is an isomorphism
\begin{displaymath}
\a1(\sfD_\sg(\gr \Lambda)) \cong \a1(k)^{\oplus n|a|}
\end{displaymath}
where $n$ is the number of simples and $a\leq0$ is the Gorenstein parameter of $\Lambda$. In particular, there is a cofibre sequence
\begin{displaymath}
\xymatrix{
\a1(k)^{\oplus n|a|} \ar[r]^-\phi & \a1(k)^{\oplus n|a|} \ar[r] & \a1(\sfD_\sg(\modu \Lambda)) \ar[r] &
}
\end{displaymath}
where $\phi$ can be written in the form
\begin{displaymath}
\phi = 
\begin{pmatrix}
-1 & 0 & \cdots & 0 & \phi_{0,a+1} \\
1 & -1 & \cdots & 0 & \phi_{-1,a+1} \\
0 & 1 & \cdots & 0 & \phi_{-2,a+1} \\
\vdots & \vdots & \cdots & \vdots & \vdots \\
0 & 0 & \cdots & -1 & \phi_{a+2, a+1} \\
0 & 0 & \cdots & 1 & \phi_{a+1,a+1} -1
\end{pmatrix}
\in \End(\a1(k)^{\oplus n|a|}).
\end{displaymath}
\end{cor}
\begin{proof}
By \cite{BurkeStevenson}*{Theorem~6.4} (which extends \cite{Orlov09}*{Corollary~2.9}) the graded singularity category has a semi-orthogonal decomposition
\begin{displaymath}
\sfD_\sg(\gr \Lambda) = (\Lambda_0(0), \ldots, \Lambda_0(a+1))
\end{displaymath}
where $a$ is the Gorenstein parameter (which is negative in this case provided the singularity category isn't trivial, so the sequence has length $a$). Since $\Lambda_0 \cong k^n$ is a semisimple algebra and $\AA^1$-homotopy invariants are functors inverting derived Morita equivalences, the isomorphism $\a1(\sfD_\sg(\gr \Lambda)) \cong \a1(k)^{\oplus n|a|}$ is a formal consequence, as noted in Remark \ref{rem:SOD}.

Since $\Lambda_0 = \Lambda/\rad(\Lambda)$ all the simples are gradable and so Theorem~\ref{thm_main1} applies to give a cofibre sequence
\begin{displaymath}
\xymatrix{
\a1(\sfD_\sg(\gr \Lambda)) \ar[rrr]^-{\a1(-1) - \id} &&& \a1(\sfD_\sg(\gr \Lambda)) \ar[r] & \a1(\sfD_\sg(\modu \Lambda)) \ar[r] & 
}
\end{displaymath}
where we have taken orbits by $(-1)$ instead of $(1)$ for convenience (which makes no difference). Using the isomorphism $\a1(\sfD_\sg(\gr \Lambda)) \cong \a1(k)^{\oplus n|a|}$ gives us a cofibre sequence of the claimed form up to verifying the description of $\phi = \a1(-1) - \id$. This description follows from noting that $(-1)$ just translates the chosen exceptional collection, except for $\Lambda_0(a+1)\mapsto \Lambda_0(a)$ which is no longer part of the collection. The final column expresses the class of $\Lambda_0(a)$ with respect to the decomposition of the Grothendieck group given by the semiorthogonal decomposition, which together with the above describes the corresponding map on $K_0$ and so, by \cite{TabuadaA1}*{Proposition~2.8}, completely determines the map $\phi$; see Remark~\ref{rem:approximation} for further explanation and intuition.
\end{proof}

\begin{rem}\label{rem:approximation}
Suppose for simplicity that $\Lambda_0 \cong k$, i.e.\ $\Lambda$ is local. The $\phi_{i,a+1}$ occurring in the final column of $\phi$ express the multiplicities occurring in the sequence of approximation triangles
\begin{displaymath}
\xymatrix{
k(a) \ar[r] & k(a)_{a+2} \ar[d] \ar[r] & \cdots \ar[r] & k(a)_{-1} \ar[r] \ar[d] & k(a)_0 \ar[r] \ar[d] & 0 \ar[d] \\
& X_{a+1} \ar[ul]^-\Sigma & & X_{-2} & X_{-1} \ar[ul]^-\Sigma & X_0 \ar[ul]^-\Sigma
}
\end{displaymath}
for $k(a)$ with respect to the full exceptional collection $(k(0), \ldots, k(a+1))$; in this diagram we have $X_i \in \thick(k(i))$, each of the triangles is distinguished, and $k(a)_{i}$ lies in $(k(0),\ldots, k(i))$. More precisely
\begin{displaymath}
\phi_{i,a+1} = [X_i] \in K_0(\thick(k(i))) \cong \ZZ,
\end{displaymath}
and these multiplicities describe $[k(a)]$ with respect to the basis of $K_0(\sfD_\sg(\gr \Lambda))$ given by $\{[k(0)], \ldots, [k(a+1)]\}$. Thus, we know the representation of the automorphism $K_0(-1)$ of $K_0(\sfD_\sg(\gr \Lambda))$ in terms of this basis and from \cite{TabuadaA1}*{Proposition~2.8} we learn that this same integer matrix describes $\a1(-1)$ with respect to the corresponding decomposition of $\a1(\sfD_\sg(\gr \Lambda))$.
\end{rem}

It is possible to give a sufficiently explicit description of the matrix $\phi$ in Corollary~\ref{cor:Gor} that one can actually compute in examples. The remainder of this section is devoted to providing this description. We start with an example, namely exterior algebras, illustrating how things work and how one can proceed with computations. We then explain, in Theorem~\ref{thm:Frobenius}, how the story presented in the example generalizes to any finite dimensional Artin-Schelter Gorenstein algebra.

\subsection{Exterior algebras}

Let $\Lambda$ be an exterior algebra on $n+1$ generators in degree $1$. This algebra is Koszul with dual $\Lambda^!$ polynomial on $n+1$ degree $1$ generators. We are then in the situation of Corollary~\ref{cor:Koszul} (and of Corollary~\ref{cor:Gor}): the BGG correspondence gives 
\begin{displaymath}
\sgr \Lambda \cong \sfD^\mathrm{b}(\coh \mathbb{P}^{n})
\end{displaymath}
and we can exploit our knowledge of projective space. Given an $\AA^1$-homotopy invariant $\a1$ we can use the triangle
\begin{displaymath}
\xymatrix{
\a1(\sfD^\mathrm{b}(\coh \PP^n)) \ar[rrr]^-{\a1(\Sigma^{-1}\str(1)\otimes \text{-}) - \id} &&& \a1(\sfD^\mathrm{b}(\coh \PP^n)) \ar[r] & \a1(\smodu \Lambda) \ar[r] & 
}
\end{displaymath}
to attempt to compute $\a1(\smodu \Lambda)$ (again we have used the inverse of the functor taken in the corollaries for the sake of convenience).

The Beilinson full exceptional collection
\begin{displaymath}
\sfD^\mathrm{b}(\coh \PP^n) = (\str(-n), \str(-n+1), \ldots, \str(-1), \str)
\end{displaymath}
implies that for any $\AA^1$-homotopy invariant $\a1$ we have
\begin{displaymath}
\a1(\sfD^\mathrm{b}(\coh \PP^n)) \cong \a1(k)^{\oplus n+1}
\end{displaymath}
and so we can rewrite our cofibre sequence as 
\begin{displaymath}
\xymatrix{
\a1(k)^{\oplus n+1} \ar[r]^-\phi & \a1(k)^{\oplus n+1} \ar[r] & \a1(\smodu \Lambda) \ar[r] &
}
\end{displaymath}
and the game is to understand $\phi$ (as in Corollary~\ref{cor:Gor} which would have gotten us to the same place, noting that the Gorenstein parameter of $\Lambda$ is $-n-1$) which is the morphism $\a1(\Sigma^{-1}\str(1)\otimes\text{-})-\id$ written with respect to the system of coordinates given by the chosen full exceptional collection. We more or less understand $\phi$ in the sense that we can write it as
\begin{displaymath}
\phi = 
\begin{pmatrix}
0 & 0 & \cdots & 0 & \psi_{-n} \\
\a1(\Sigma^{-1}) & 0 & \cdots & 0 & \psi_{-n+1} \\
0 & \a1(\Sigma^{-1}) & \cdots & 0 & \psi_{-n+2} \\
\vdots & \vdots & \cdots & \vdots & \vdots \\
0 & 0 & \cdots & 0 & \psi_{-1} \\
0 & 0 & \cdots & \a1(\Sigma^{-1}) & \psi_{0} 
\end{pmatrix}
- \mathrm{Id}_{n+1},
\end{displaymath}
the main `difficulty' being the computation of the $\psi_i$ which are the multiplicities for $\Sigma^{-1}\str(1)$. Indeed we know, by additivity for $\AA^1$-homotopy invariants \cite{TabuadaA1}*{Proposition~2.5}, that $\a1(\Sigma) = -1$ and so only the last column needs to be computed.

It turns out this is also relatively straightforward and doesn't depend on $\a1$. In fact, as in \cite{TabuadaA1}*{Proposition~2.8}, this comes down to computing the filtration by triangles for $\Sigma^{-1}\str(1)$ as indicated in Remark \ref{rem:approximation}. This is essentially given by (the desuspension of) the Koszul complex
\begin{displaymath}
0 \to \str(-n) \to \str(-n+1)^{\oplus\binom{n+1}{n}} \to \cdots \to \str^{\oplus\binom{n+1}{1}} \to \str(1) \to 0
\end{displaymath}
and so we see that $\psi_{-i} = (-1)^{i+1}\binom{n+1}{i+1}$ for $0\leq i \leq n$.

Thus $\a1(\smodu \Lambda)$ is the cone of the endomorphism
\begin{displaymath}
\phi = 
\begin{pmatrix}
-1 & 0 & \cdots & 0 & (-1)^{n+1}\binom{n+1}{n+1} \\
-1 & -1 & \cdots & 0 & (-1)^{n}\binom{n+1}{n} \\
0 & -1 & \cdots & 0 & (-1)^{n-1}\binom{n+1}{n-1} \\
\vdots & \vdots & \cdots & \vdots & \vdots \\
0 & 0 & \cdots & -1 & \binom{n+1}{2} \\
0 & 0 & \cdots & -1 & -\binom{n+1}{1}-1
\end{pmatrix}
\end{displaymath}
of $\a1(k)^{\oplus n+1}$ and, as luck would have it, this cone is legitimately computable. Indeed, to compute the cone we can replace $\phi$ by its Smith normal form. An easy computation (the reader who is rightly suspicious of such statements will be reassured that we give an abstract justification for this computation in the next section, see the proof of Theorem~\ref{thm:Frobenius}) shows that the Smith normal form is
\begin{displaymath}
\phi' =
\begin{pmatrix}
1 & 0 & 0 &\cdots & 0 \\
0 & 1 & 0 &\cdots & 0 \\
0 & 0 & 1 & \cdots & 0 \\
\vdots & \vdots & \vdots & \cdots & \vdots \\
0 & 0 & 0 &\cdots & \det \phi
\end{pmatrix}
\end{displaymath}
where $\det\phi = \dim \Lambda = 2^{n+1}$. Hence we have proved:

\begin{thm}\label{thm:exterior}
Let $\Lambda$ denote the exterior algebra on $n+1$ generators. Then for any $\AA^1$-homotopy invariant $\a1$ we have
\begin{displaymath}
\a1(\smodu \Lambda) \cong \cone(\a1(k) \xymatrix{ \ar[rr]^-{\boldsymbol{\cdot} 2^{n+1}} && } \a1(k)).
\end{displaymath}
\end{thm}

\begin{ex}
We could take $\a1$ to be Weibel's homotopy K-theory $KH$. In this case we recover, for example, the computation that 
\begin{displaymath}
K_0(\smodu \Lambda) \cong KH_0(\smodu \Lambda) \cong \ZZ/2^{n+1}\ZZ.
\end{displaymath}
\end{ex}


\subsection{Connected graded self-injective algebras}

We now indicate how the argument given for exterior algebras generalizes. Let $\Lambda$ be a finite dimensional basic Artin-Schelter Gorenstein $k$-algebra with $\Lambda_0 = \Lambda/\rad(\Lambda)$, where $\rad(\Lambda)$ denotes the ungraded radical, and $n$ simples all of which we assume are $1$-dimensional. In particular, $\Lambda$ is self-injective. Theorem~\ref{thm:exterior} naturally extends to this setting. In order to prove this we first need a technical lemma along the lines of \cite{TabuadaA1}*{Corollary~1.6}.

\begin{lem}\label{lem:nicesave}
There is an isomorphism $K_0(\smodu \Lambda) \cong KH_0(\smodu \Lambda)$.
\end{lem}
\begin{proof}
There is a natural comparison map $K\to KH$. Using this comparison and the sequence
\begin{displaymath}
\xymatrix{
\sgr \Lambda \ar[r]^-{(1)} & \sgr \Lambda \ar[r]^-{\pi} & \smodu \Lambda \cong \Perf(\sgr \Lambda / (1))
}
\end{displaymath}
we get a commutative diagram
\begin{displaymath}
\xymatrix{
K_0(\sgr \Lambda) \ar[rr]^-{K_0((1)) - 1} \ar[d]^-\wr & &K_0(\sgr \Lambda) \ar[r]^-{K_0(\pi)} \ar[d]^-\wr & K_0(\smodu \Lambda) \ar[d] & \\
KH_0(\sgr \Lambda) \ar[rr]^-{KH_0((1)) - 1} && KH_0(\sgr \Lambda) \ar[r]^-{KH_0(\pi)} & KH_0(\smodu \Lambda) \ar[r] & 0 \\
}
\end{displaymath}
where the top composite is zero, and the first two vertical maps are isomorphisms since as in Corollary~\ref{cor:Gor} the category $\sgr \Lambda$ has a full exceptional collection so
\begin{displaymath}
KH(\sgr \Lambda) \cong KH(k)^{\oplus n\vert a \vert} \cong K(k)^{\oplus n\vert a \vert} \cong K(\sgr \Lambda).
\end{displaymath}
by derived Morita invariance and additivity (cf.\ \cite{TabuadaFT}*{Proposition~2.3} and \cite{WeibelKBook}*{Example~IV.12.5.1}). It also follows that the bottom row is exact since $0 = K_{-1}(k)^{\oplus n\vert a \vert} \cong KH_{-1}(\sgr \Lambda)$. Furthermore, the map $K_0(\pi)$ is surjective: the classes of the simple modules generate $K_0(\smodu \Lambda)$ and these classes are in the image of $K_0(\pi)$ as we have assumed the simples are all gradable. It then follows from a diagram chase that the third vertical map is an isomorphism as claimed.

\end{proof}

We denote by $C_\Lambda$ the Cartan matrix of $\Lambda$. It is the $n\times n$ integer matrix, where $n$ is the number of simples, whose entry in position $(i,j)$ is $\dim_k\Hom_\Lambda(P_i,P_j)$ for some fixed, arbitrary, order on the simple modules, where $P_i$ is the projective with top the $i$th simple.

\begin{thm}\label{thm:Frobenius}
For any $\AA^1$-homotopy invariant $\a1$ we have
\begin{displaymath}
\a1(\smodu \Lambda) \cong \cone(\a1(k)^{\oplus n} \xymatrix{\ar[r]^-{C_\Lambda} &} \a1(k)^{\oplus n})
\end{displaymath}
where $C_\Lambda$ is the Cartan matrix.
\end{thm}
\begin{proof}
By Corollary~\ref{cor:Gor} we know there is a cofibre sequence 
\begin{displaymath}
\xymatrix{
\a1(k)^{\oplus n|a|} \ar[r]^-\phi & \a1(k)^{\oplus n|a|} \ar[r] & \a1(\smodu \Lambda) \ar[r] &
}
\end{displaymath}
where $\phi$ has the form indicated in the Corollary (and as in the last section and \cite{TabuadaA1}*{Proposition~2.8}). In particular, $\phi$ is an integer matrix which does not depend on $\a1$.

Taking $\a1$ to be homotopy K-theory and looking at $0$th homotopy groups we get an exact sequence
\begin{displaymath}
\ZZ^{\oplus n|a|} \stackrel{\phi}{\to} \ZZ^{\oplus n|a|} \to KH_0(\smodu \Lambda) \to 0
\end{displaymath}
as in the proof of the previous lemma. Moreover, by said lemma there is an isomorphism $KH_0(\smodu \Lambda) \cong K_0(\smodu \Lambda)$ and we know that $K_0(\smodu \Lambda) \cong \coker(C_\Lambda)$ by \cite{TW91}. This is only possible if the Smith normal form of $\phi$ is
\begin{displaymath}
\begin{pmatrix}
\mathrm{Id}_{n(\vert a\vert -1)} & 0 \\
0 & S_\Lambda
\end{pmatrix}
\end{displaymath}
up to signs, where $S_\Lambda$ is the Smith normal form of $C_\Lambda$.
Since $\phi$ is independent of the invariant $\a1$ and we can use its Smith normal form to compute the cone in the cofibre sequence computing $\a1(\smodu \Lambda)$ the result follows.
\end{proof}

In particular, the $\AA^1$-invariants of the stable categories of such rings only depend on the Cartan matrix, i.e.\ on the dimension vectors of the projectives.

\begin{rem}
This extends \cite{TW91}*{Proposition~1}, which computes the Grothendieck group of the stable category, to arbitrary $\AA^1$-homotopy invariants.
\end{rem}



This allows us to perform various explicit computations over a finite field where we know the homotopy K-theory explicitly by the work of Quillen \cite{Quillen}*{Theorem~8}. For instance, we deduce the following corollary.

\begin{cor}
Let $\Lambda$ be a self-injective algebra over $\FF_p$ admitting a connected grading and of dimension $p^n$ for some $n\geq 1$. Then
\begin{displaymath}
KH_i(\smodu \Lambda) = \left\{\begin{array}{lr}
\ZZ/p^n\ZZ & \text{if } i=0; \\
0 & \text{if } i\geq 1.
\end{array}
\right.
\end{displaymath} 
In particular the inclusion $\sfD^\mathrm{perf}(\Lambda) \to \sfD^\mathrm{b}(\modu \Lambda)$ induces isomorphisms
\begin{displaymath}
KH_i(\Lambda) \stackrel{\sim}{\to} KH_i(\sfD^\mathrm{b}(\modu \Lambda))
\end{displaymath}
for $i\geq 1$.
\end{cor}
\begin{proof}
Taking $\a1 = KH$ in Theorem~\ref{thm:Frobenius} and taking homotopy groups immediately yields the first computation by inspecting the resulting long exact sequence. Indeed, by \cite{WeibelKBook}*{12.3.1} the homotopy K-theory of $\FF_p$ agrees with the usual algebraic K-theory which was computed by Quillen \cite{Quillen}, so we just observe that $p^n$ is invertible in $K_{2i-1}(\FF_p) \cong \ZZ/(p^i-1)\ZZ$ for $i\geq 1$ and $K_{2i}(\FF_p)$ is zero for $i\geq 1$. The second statement then follows from the long exact sequence for the localization sequence
\begin{displaymath}
\sfD^\mathrm{perf}(\Lambda) \to \sfD^\mathrm{b}(\Lambda) \to \smodu \Lambda.
\end{displaymath}
\end{proof}

\begin{ex}
This applies for instance to the group algebra of $E^r = (\ZZ/p\ZZ)^{\oplus r}$ over $\FF_p$, showing that $\smodu \FF_pE^r$ has no higher homotopy K-theory. In this case one can interpret the second statement of the corollary as computing the homotopy K-theory of cochains on the classifying space of $E^r$:
\begin{displaymath}
KH_i(C^*(BE^r;\FF_p)) \cong K_i(\FF_p) \quad \text{ for all } i\geq 0.
\end{displaymath}
\begin{proof}
We first note that, by identifying $C^*(BE^r;\FF_p)$ with $\Ext^*_{kE^r}(k,k)$, there is an equivalence (with a DG-enhancement) $\Perf(C^*(BE^r;\FF_p)) \cong \sfD^\mathrm{b}(\FF_pE^r)$. By the corollary we thus have isomorphisms 
\begin{displaymath}
KH_i(C^*(BE^r;\FF_p)) \cong KH_i(\sfD^\mathrm{b}(\FF_pE^r)) \cong KH_i(\FF_pE^r) \quad \text{ for all } i\geq 0.
\end{displaymath}
It remains to note that, as $KH$ doesn't detect nilpotent extensions and agrees with usual $K$-theory for regular rings (see for instance \cite{WeibelKBook} Corollary~IV.12.5 and IV.12.3.1 respectively), we have $KH_i(\FF_pE^r) \cong KH_i(\FF_p) \cong K_i(\FF_p)$ for all $i\geq 0$.
\end{proof}
\end{ex}

\begin{rem}
One can use the theorem to produce many $\AA^1$\emph{-homotopy phantoms}, i.e.\ DG-categories all of whose $\AA^1$-invariants are trivial. Indeed, the stable category of any suitable $\Lambda$ with invertible Cartan matrix will do. This will be pursued in future work.
\end{rem}



\section{Phantoms from cluster theory}

In this section we observe that in Dynkin type $A_{2n}$ cluster categories have trivial $\AA^1$-homotopy invariants, i.e.\ they are `$\AA^1$-homotopy phantoms'. This is straightforward from work of Tabuada, who gave an expression for the $\AA^1$-homotopy invariants of cluster categories of finite quivers without oriented cycles in \cite{TabuadaA1}*{Corollary~2.11}. 

However, it isn't made explicit there that for $A_{2n}$ the stars align so that these invariants always vanish. We feel this is worth noting as the phantoms occurring in algebraic geometry are notoriously slippery, see \cite{Sosna} for an introduction and several open questions, yet here we have a very concrete family of categories with trivial homotopy K-theory just lying around in the representation theorist's toolbox.

We now sketch the relevant setup (which proceeds, as one might expect, following \cite{TabuadaA1} and what we have done above).

Let us fix some field $k$ and consider $\sfD^\mathrm{b}(\modu kA_n)$, where for simplicity we will always think of $A_n$ with the linear orientation
\begin{displaymath}
1 \to 2 \to \cdots \to n.
\end{displaymath}
Of course, the derived category is independent of the orientation so this is purely a matter of convenience. The simples $S_i$ form a generating set for $\sfD^\mathrm{b}(\modu kA_n)$ and, in fact, $(S_n,\ldots,S_1)$ is a full exceptional collection. Thus, if $\a1$ is any $\AA^1$-homotopy invariant, we have
\begin{displaymath}
\a1(\sfD^\mathrm{b}(\modu kA_n)) \cong \a1(\thick(S_1)) \oplus \cdots \oplus \a1(\thick(S_n)) \cong \a1(k)^{\oplus n}.
\end{displaymath}

The ($2$-)cluster category of $A_n$ over $k$ is obtained by taking orbits by $\Sigma\tau^{-1}$
\begin{displaymath}
\sfC_{A_n} = \sfD^\mathrm{b}(\modu kA_n) / \Sigma\tau^{-1}.
\end{displaymath}
So we have, by \cite{TabuadaA1}*{Theorem~1.5}, an identification
\begin{displaymath}
\a1(\sfC_{A_n}) = \cone(\a1(\Sigma)\a1(\tau^{-1}) - 1).
\end{displaymath}
By additivity for $\AA^1$-homotopy invariants \cite{TabuadaA1}*{Proposition~2.5} we know $\a1(\Sigma) = -1$ and so it just remains to write down $\a1(\tau^{-1})$ in terms of the decomposition coming from the simples. We know that
\begin{displaymath}
\tau^{-1} S_i = S_{i+1}
\end{displaymath}
for $i\leq n-1$ and $\tau^{-1} S_n = \Sigma P_n$. Thus $\a1(\tau^{-1})$ can be represented by the $n\times n$ matrix
\begin{displaymath}
\begin{pmatrix}
0 & 0 & \cdots & 0 & -1 \\
1 & 0 & \cdots & 0 & -1 \\
0 & 1 & \cdots & 0 & -1 \\
\vdots & \vdots & \cdots & \vdots & \vdots \\
0 & 0 & \cdots & 1 & -1
\end{pmatrix}
\end{displaymath}
with respect to the decomposition coming from the exceptional collection $(S_n,\ldots,S_1)$.

Our rather modest observation is the following lemma.

\begin{lem}
If $n$ is even then the matrix
\begin{displaymath}
\a1(\Sigma\tau^{-1}) -1 = -\a1(\tau^{-1}) - 1 =
\begin{pmatrix}
-1 & 0 & \cdots & 0 & 1 \\
-1 & -1 & \cdots & 0 & 1 \\
0 & -1 & \cdots & 0 & 1 \\
\vdots & \vdots & \cdots & \vdots & \vdots \\
0 & 0 & \cdots & -1 & 0
\end{pmatrix}
\end{displaymath}
has determinant $1$. In particular, it is an automorphism of
\begin{displaymath}
\a1(\sfD^\perf(\modu kA_n)) \cong \a1(k)^{\oplus n}.
\end{displaymath}
\end{lem}
\begin{proof}
We will in fact prove that the determinant of the above matrix, which we denote for the duration of the proof by $\phi_n$, is $1$ if $n$ is even and $0$ if $n$ is odd. We proceed by induction on $n$ starting with the case $n=2$, i.e.\
\begin{displaymath}
\phi_2 =
\begin{pmatrix}
-1 & 1 \\
-1 & 0
\end{pmatrix}
\end{displaymath}
where one just observes the determinant is $1$. Assume then that the claim holds for $n-1 \geq 2$. By taking the Laplace expansion along the first row of $\phi_n$ we see
\begin{align*}
\det(\phi_{n}) &= (-1)\det(\phi_{n-1}) + (-1)^{n+1}\det(X_n) \\
&= -\det(\phi_{n-1}) + (-1)^{2n} \\
&= -\det(\phi_{n-1}) + 1,
\end{align*}
where $X_n$ is the $(n-1) \times (n-1)$-matrix
\begin{displaymath}
X_n = 
\begin{pmatrix}
-1 & -1 & 0 & \cdots & 0 \\
0 & -1 & -1 & \cdots & 0  \\
\vdots & \vdots & \cdots & \vdots & \vdots \\
0 & 0 & \cdots & -1 & -1 \\
0 & 0 & \cdots & 0 & -1
\end{pmatrix}.
\end{displaymath}
Thus $\det(\phi_{n})$ is $-1 + 1 = 0$ if $n$ is odd and is $0 + 1 = 1$ if $n$ is even.
\end{proof}

This has the following rather striking consequence.

\begin{thm}\label{thm:phantom}
If $n$ is even then for any $\AA^1$-homotopy invariant $\a1$ we have
\begin{displaymath}
\a1(\sfC_{A_n}) = 0.
\end{displaymath}
\end{thm}
\begin{proof}
We know $\a1(\sfC_{A_n}) = \cone(\a1(\Sigma)\a1(\tau^{-1}) - 1)$. By the lemma the map $\a1(\Sigma\tau^{-1}) - 1$ is invertible when $n$ is even and so its cone vanishes.
\end{proof}


Thus we have had some manner of ``$\AA^1$-phantoms'' under our noses for some time. As a corollary to the theorem we deduce another surprising fact. Let $\Gamma_{n}$ denote the Ginzburg DG algebra associated to $A_n$ as in \cite{Ginzburg06}*{Section~4.2}. We can then consider the DG category of perfect complexes over $\Gamma_n$, denoted $\Perf(\Gamma_n)$, and its full DG subcategory consisting of those complexes with finite dimensional total cohomology, denoted $\Perf_{\mathrm{fd}}(\Gamma_n)$.

By \cite{Amiot09}*{Corollary~3.12} there is a quasi-equivalence
\begin{displaymath}
\Perf(\Gamma_n) / \Perf_{\mathrm{fd}}(\Gamma_n) \cong \sfC_{A_n}.
\end{displaymath}
Combining this identification of the quotient with \ref{thm:phantom} gives the following computation.

\begin{cor}
If $n$ is even the inclusion $i\colon\Perf_{\mathrm{fd}}(\Gamma_n) \to \Perf(\Gamma_n)$ induces, for any $\AA^1$-homotopy invariant $\a1$, an isomorphism
\begin{displaymath}
\a1(i)\colon \a1(\Perf_{\mathrm{fd}}(\Gamma_n)) \stackrel{\sim}{\to} \a1(\Gamma_n).
\end{displaymath}
\end{cor}
\begin{proof}
By definition $\a1$ applied to the localization sequence
\begin{displaymath}
\Perf_{\mathrm{fd}}(\Gamma_n) \to \Perf(\Gamma_n) \to \sfC_{A_n}
\end{displaymath}
gives a triangle
\begin{displaymath}
\a1(\Perf_{\mathrm{fd}}(\Gamma_n)) \to \a1(\Perf(\Gamma_n)) \to \a1(\sfC_{A_n}) \to \Sigma \a1(\Perf_{\mathrm{fd}}(\Gamma_n)).
\end{displaymath}
By the theorem the third object in this triangle is trivial, from which the assertion follows immediately.
\end{proof}

We note that, unlike $\Perf(\Gamma_n)$ and $\sfC_{A_n}$, the category $\Perf_{\mathrm{fd}}(\Gamma_n)$ is \emph{not} smooth. Thus the corollary gives an explicit example of a smooth DG category which cannot be distinguished from a non-smooth DG subcategory by any $\AA^1$-homotopy invariant.

\begin{rem}
The above phenomenon is somewhat special to the case of Dynkin type $A$. In \cite{BKL} the Grothendieck groups of cluster categories are described for algebras to which Keller's result on pretriangulated orbit categories \cite{KellerOrbit} applies. They show in \cite{BKL}*{Proposition~3.5} that for Dynkin quivers the Grothendieck group vanishes if and only if the quiver is of type $A_n$ or $E_n$ with $n$ even. Other quivers are also discussed, for instance they show that the Grothendieck group can never vanish for a canonical algebra with only odd weights (\cite{BKL}*{Theorem~1.3}).
\end{rem}



\bibliography{greg_bib}

\begin{bibdiv}
\begin{biblist}

\bib{AHK-theory}{article}{
      author={Antieau, Benjamin},
      author={Heller, Jeremiah},
       title={Some remarks on topological {$K$}-theory of dg categories},
        date={2018},
     journal={Proceedings of the American Mathematical Society},
      volume={146},
      number={10},
       pages={4211\ndash 4219},
}

\bib{Amiot09}{article}{
      author={Amiot, Claire},
       title={Cluster categories for algebras of global dimension 2 and quivers
  with potential},
        date={2009},
     journal={Ann. Inst. Fourier (Grenoble)},
      volume={59},
      number={6},
       pages={2525\ndash 2590},
}

\bib{BGG}{article}{
      author={Bernstein, I.~N.},
      author={Gel’fand, I.~M.},
      author={Gel’fand, S.~I.},
       title={Algebraic vector bundles on $\mathbb{P}^n$ and problems of linear
  algebra},
        date={1978},
     journal={Funktsional. Anal. i Prilozhen},
      volume={12},
       pages={66\ndash 67},
}

\bib{BKL}{article}{
      author={Barot, M.},
      author={Kussin, D.},
      author={Lenzing, H.},
       title={The {G}rothendieck group of a cluster category},
        date={2008},
     journal={J. Pure Appl. Algebra},
      volume={212},
      number={1},
       pages={33\ndash 46},
}

\bib{BurkeStevenson}{incollection}{
      author={Burke, Jesse},
      author={Stevenson, Greg},
       title={The derived category of a graded {G}orenstein ring},
        date={2015},
   booktitle={Commutative algebra and noncommutative algebraic geometry. {V}ol.
  {II}},
      series={Math. Sci. Res. Inst. Publ.},
      volume={68},
   publisher={Cambridge Univ. Press, New York},
       pages={93\ndash 123},
      review={\MR{3496862}},
}

\bib{Drinfeld}{article}{
      author={Drinfeld, Vladimir},
       title={D{G} quotients of {DG} categories},
        date={2004},
     journal={J. Algebra},
      volume={272},
      number={2},
       pages={643\ndash 691},
}

\bib{Ginzburg06}{article}{
      author={Ginzburg, Victor},
       title={Calabi-{Y}au algebras},
        date={2006},
     journal={arXiv:0612139},
}

\bib{KellerOrbit}{article}{
      author={Keller, Bernhard},
       title={On triangulated orbit categories},
        date={2005},
     journal={Doc. Math.},
      volume={10},
       pages={551\ndash 581},
}

\bib{KellerDG}{incollection}{
      author={Keller, Bernhard},
       title={On differential graded categories},
        date={2006},
   booktitle={International {C}ongress of {M}athematicians. {V}ol. {II}},
   publisher={Eur. Math. Soc., Z\"{u}rich},
       pages={151\ndash 190},
}

\bib{KellerQ}{article}{
      author={Keller, Bernhard},
       title={On the cyclic homology of exact categories},
        date={1999},
     journal={J. Pure Appl. Algebra},
      volume={136},
      number={1},
       pages={1\ndash 56},
}

\bib{KMV}{article}{
      author={Keller, Bernhard},
      author={Murfet, Daniel},
      author={Van~den Bergh, Michel},
       title={On two examples by {I}yama and {Y}oshino},
        date={2011},
     journal={Compos. Math.},
      volume={147},
      number={2},
       pages={591\ndash 612},
}

\bib{KrauseStable}{article}{
      author={Krause, Henning},
       title={The stable derived category of a {N}oetherian scheme},
        date={2005},
     journal={Compos. Math.},
      volume={141},
      number={5},
       pages={1128\ndash 1162},
}

\bib{VillaSaorin}{article}{
      author={Mar{t\'\i}nez~Villa, Roberto},
      author={Saor{\'\i n}, Manuel},
       title={Koszul equivalences and dualities},
        date={2004},
     journal={Pacific J. Math.},
      volume={214},
      number={2},
       pages={359\ndash 378},
}

\bib{Orlov09}{incollection}{
      author={Orlov, Dmitri},
       title={Derived categories of coherent sheaves and triangulated
  categories of singularities},
        date={2009},
   booktitle={Algebra, arithmetic, and geometry: in honor of {Y}u. {I}.
  {M}anin. {V}ol. {II}},
      series={Progr. Math.},
      volume={270},
   publisher={Birkh\"auser Boston, Inc., Boston, MA},
       pages={503\ndash 531},
}

\bib{PPQA}{book}{
      author={Polishchuk, Alexander},
      author={Positselski, Leonid},
       title={Quadratic algebras},
      series={University Lecture Series},
   publisher={American Mathematical Society, Providence, RI},
        date={2005},
      volume={37},
}

\bib{Quillen}{article}{
      author={Quillen, Daniel},
       title={On the cohomology and {$K$}-theory of the general linear groups
  over a finite field},
        date={1972},
     journal={Ann. of Math. (2)},
      volume={96},
       pages={552\ndash 586},
}

\bib{Sosna}{article}{
      author={Sosna, Pawel},
       title={Some remarks on phantom categories and motives},
        date={2015},
     journal={preprint arXiv:1511.07711},
}

\bib{TabuadaFT}{article}{
      author={Tabuada, Gonçalo},
       title={The fundamental theorem via derived {M}orita invariance,
  localization, and {$\mathbb{A}^1$}-homotopy invariance},
        date={2012},
     journal={J. K-theory},
      volume={9},
      number={3},
       pages={407–420},
}

\bib{TabuadaA1}{article}{
      author={Tabuada, Gon\c{c}alo},
       title={{$\mathbb{A}^1$}-homotopy invariants of dg orbit categories},
        date={2015},
     journal={J. Algebra},
      volume={434},
       pages={169\ndash 192},
}

\bib{TW91}{article}{
      author={Tachikawa, Hiroyuki},
      author={Wakamatsu, Takayoshi},
       title={{C}artan matrices and {G}rothendieck groups of stable
  categories},
        date={1991},
        ISSN={0021-8693},
     journal={J. Algebra},
      volume={144},
      number={2},
       pages={390 \ndash  398},
  url={http://www.sciencedirect.com/science/article/pii/002186939190111K},
}

\bib{WeibelKBook}{book}{
      author={Weibel, Charles~A.},
       title={The {$K$}-book},
      series={Graduate Studies in Mathematics},
   publisher={American Mathematical Society, Providence, RI},
        date={2013},
      volume={145},
}

\bib{Yamaura}{article}{
      author={Yamaura, Kota},
       title={Realizing stable categories as derived categories},
        date={2013},
     journal={Adv. Math.},
      volume={248},
       pages={784\ndash 819},
}

\end{biblist}
\end{bibdiv}

\end{document}